\documentclass{amsart}

\usepackage{amssymb,amsmath,amscd,graphicx,amsthm,color,enumerate}

\usepackage{pst-node}
\usepackage{tikz-cd}

\newtheorem{theorem}{Theorem}
\newtheorem{lemma}[theorem]{Lemma}
\newtheorem{proposition}[theorem]{Proposition}

\theoremstyle{definition}

\theoremstyle{remark}
\newtheorem{remark}[theorem]{Remark}
\newtheorem{example}[theorem]{Example}
\numberwithin{equation}{section}
\numberwithin{theorem}{section}

\def\A{{\mathcal A}}
\def\AA{{\mathbb A}}

\def\C{{\mathbb C}}
\def\CC{{\mathcal C}}

\def\G{{\mathcal G}}
\def\GCC{{\G\CC}}

\def\P{{\mathcal P}}

\def\UU{\overline{\A}}

\def\Z{{\mathbb Z}}

\def\bfG{{\mathbf \Gamma}}
\def\bfGr{\bfG^{{\rm r}}}
\def\bfGc{\bfG^{{\rm c}}}
\def\gammar{\gamma^{\rm r}}
\def\gammac{\gamma^{\rm c}}

\def\ttf{{\tt f}}
\def\ttg{{\tt g}}
\def\tth{{\tt h}}
\def\ttp{{\tt p}}
\def\ttq{{\tt q}}
\def\ttt{{\tt t}}
\def\tty{{\tt y}}

\def\ttF{{\tt F}}
\def\ttG{{\tt G}}
\def\ttH{{\tt H}}
\def\ttP{{\tt P}}
\def\ttT{{\tt T}}

\def\Gr{\operatorname{Gr}}

\def\Id{{\operatorname {Id}}}

\def\Mat{\operatorname{Mat}}

\def\Poi{{\{\cdot,\cdot\}}}
\def\Poip{\Poi_{\rm p}}

\def\nar{\setlength\arraycolsep{2pt}}

\def\:{{:\ }}

\begin{document}

\title{Regular pullback of generalized cluster structures}

\author{Misha Gekhtman}
\address{Department of Mathematics, University of Notre Dame, Notre Dame,
IN 46556}
\email{mgekhtma@nd.edu}

\author{Michael Shapiro}
\address{Department of Mathematics, Michigan State University, East Lansing,
MI 48823}
\email{mshapiro@math.msu.edu}

\author{Alek Vainshtein}
\address{Department of Mathematics \& Department of Computer Science, University of Haifa, Haifa,
Mount Carmel 31905, Israel}
\email{alek@cs.haifa.ac.il}

\begin{abstract}
We consider the problem of lifting a regular cluster structure on a quasi-affine variety to the ambient affine space and a similar problem of defining a regular pullback of a regular cluster structure under a dominant rational map between affine spaces. We provide sufficient conditions for the existence of the corresponding object, called an almost-cluster structure, study its combinatorics, compatible Poisson bracket and the corresponding upper cluster algebra.  
\end{abstract}


\maketitle

\medskip

\section{Introduction}

 In recent years, we studied cluster structures on simple Lie groups and their homogeneous spaces that are compatible with (homogeneous) Belavin--Drinfeld Poisson--Lie brackets, see, e.g.,~\cite{GSVple} and references therein. While working toward a general construction, we encountered in~\cite{currentGSV} the following two situations: (i) a cluster-like structure on a Lie group arises as the pullback of a compatible (generalized) cluster structure along a dominant rational map from a Poisson affine space; (ii) a cluster-like structure on a Poisson affine space arises as the pullback of a compatible (generalized) cluster structure on a Lie group embedded in this affine space as a quasi-affine variety.
It turns out that these two procedures share many common features, and 
we expect that this pullback mechanism is useful well beyond the particular example considered in~\cite{currentGSV}. For this reason, we present here the general construction of a \emph{regular pullback\/}.

To explain the idea, we begin with a representative special case. Let
$\Psi:\AA^{r}\dashrightarrow \AA^{s}$ be a dominant rational map, and assume that $\AA^{s}$ carries a \emph{regular\/} generalized cluster structure $\CC$, meaning that every cluster variable in every cluster is a regular function on 
$\AA^s$. Fix affine coordinate systems $x_1,\dots,x_r$ on $\AA^{r}$ and $y_1,\dots,y_s$ on $\AA^{s}$, then every cluster variable 
$f_l$ can be rewritten as a rational function in $x_1,\dots,x_r$, and is not regular on $\AA^r$.
 Our goal is to construct on $\AA^{r}$ an analogue  of a regular generalized cluster structure, which we denote
 $\widehat{\CC}$.
Instead of looking at $\Psi^*f_l$, we take for cluster variables their numerators, which are again regular functions. Additionally, new frozen cluster variables are irreducible factors from the denominators of $\Psi^*f_l$. 
The exchange relations in $\widehat{\CC}$ are obtained by pulling back the exchange relations in $\CC$, rewriting the result over a common denominator, and then taking the numerator.

The key point is that the resulting pullbacks of generalized exchange relations may differ from standard generalized cluster relations. In both settings, the right-hand side is a polynomial whose monomials are linearly ordered. In a usual generalized cluster relation, the lowest and highest monomials are coprime; in a pullback relation, these extreme monomials need not be coprime.
Although in both cases the mutation of relations can be encoded via mutations of an exchange quiver, the quiver mutation rules arising from cluster pullback can be substantially more involved, see Section~\ref{combin} for details.

The construction of cluster pullback is given in Section~\ref{prelim}. In Section~\ref{maincond} we formulate a necessary condition for the pullback to be well-defined. Section~\ref{combin} provides a combinatorial description of the mutation rules. In Section~\ref{poisson} we discuss compatibility of Poisson brackets on $\AA^{r}$ and $\AA^{s}$. Finally, in Section~\ref{ucla} we give sufficient conditions ensuring that the pullback $\widehat{\A}$ of the upper cluster algebra is regular and complete, that is, $\widehat{\A}=\C[\AA^{r}]$.

Our research was supported in part by the NSF research grants DMS \#2100785 (M.~G.) and  DMS \#2100791 (M.~S.), and ISF grant \#2848/25 (A.~V.).

\section{General question} \label{prelim}

Let $G=\{g_i\}_{i\in I}$ be a finite collection of distinct irreducible polynomials on an $s$-dimensional affine space $\AA^s$ with coordinates $y_1,\dots,y_s$. Each $g_i$ defines an irreducible divisor $D_i=\{g_i=0\}$. Let $V=\AA^s\setminus D_V$ with $D_V=\cup_{i\in I}D_i$ be the quasi-affine variety defined by the collection $\{D_i\}_{i\in I}$, so that any
regular function $f$ on $V$ can be written in a unique way as 
\begin{equation}\label{regonV}
f=\hat f\prod_{i\in I}g_i^{-\lambda_i(f)}
\end{equation}
 where $\hat f\in\C[y_1,\dots,y_s]$ is a  polynomial not divisible by any of $g_i$ and $\lambda_i(f)\in\Z$, $i\in I$.
 Note that here and in what follows we allow for negative values of $\lambda_i(f)$.

Further, let $r\ge s$ and let $\Psi\:\AA^r\dashrightarrow\AA^s$ be a dominant rational map from an $r$-dimensional affine space $\AA^r$ with coordinates $x_1,\dots,x_r$ to $\AA^s$ satisfying the following condition: there exists a family of $r-s$ irreducible algebraically independent polynomials $\ttH=\{\tth_j\}_{j\in J}$ in $\C[x_1,\dots,x_r]$
such that for every $y_m$, $m\in[1,s]$, one can write the pullback of $y_m$ 
as $\Psi^*y_m=\tty_m\prod_{j\in J}\tth_j^{-\kappa_{j}(y_m)}$ with $\tty_m\in\C[x_1,\dots.x_r]$ not divisible by any of $\tth_j$ and $\kappa_{j}(y_m)\in\Z$, $j\in J$; (it is possible that for some $j\in J$, $\kappa_j(y_m)=0$ for all 
$m\in[1,s]$).
Consequently, for any polynomial $p\in\C[y_1,\dots,y_s]$ there exists a unique polynomial $\ttp\in\C[x_1,\dots,x_r]$ not divisible by any of $\tth_j$ such that
\begin{equation}\label{polypull}
\Psi^*p=\ttp\prod_{j\in J}\tth_j^{-\kappa_{j}(p)}
\end{equation}
 with $\kappa_{j}(p)\in\Z$, $j\in J$. 
In particular, such a polynomial for $g_i$ is denoted $\ttg_i$, $i\in I$.
Polynomials $\{\ttg_i\}_{i\in I}$ and $\{\tth_j\}_{j\in J}$ are called {\it distinguished}.
In what follows it will be convenient to refer to the family of distinguished polynomials as 
$\{\ttq_k\}_{k\in K}$ with $K=I\cup J$, so that $\ttq_k=\ttg_k$ for $k\in I$ and $\ttq_k=\tth_k$ for $k\in J$. 

Assume that $\CC$ is a regular (generalized) cluster structure on $V$, and $F\cup G=\{f_l\}_{l=1}^s$ is an initial cluster for $\CC$ that includes $\{g_i\}_{i\in I}$ as frozen variables (there may be other frozen variables in $F$ as well).  It follows from~\eqref{regonV},~\eqref{polypull} that for $f_l\in F$ 
\begin{equation}\label{regpull}
\Psi^*f_l=\hat\ttf_l\prod_{k\in K}\ttq_k^{-\lambda_k(f_l)}.
\end{equation}
Here $\hat\ttf_l\in\C[x_1,\dots,x_r]$ is not divisible by any of $\ttq_k$, exponents $\lambda_k(f_l)$ for $k\in I$ are defined 
in~\eqref{regonV}, and exponents $\lambda_k(f_l)\in\Z$ for $k\in J$ can be written
via $\lambda_i(f_l)$, $\kappa_j(g_i)$, and $\kappa_j(\hat f_l)$ defined above; however, the explicit 
formula for $\lambda_k(f_l)$ is not essential for further considerations.

Our goal is to study the behavior of functions $\hat\ttf_l$ under transitions between the clusters of $\CC$ directly,
without tracing the behavior of original functions $f_l$.
 We will prove that under certain mild restrictions, given an initial seed $\Sigma=(Q,F\cup G,\P)$ for $\CC$, one can define a seed $\Psi^*\widehat{\Sigma}=(\Psi^*\widehat Q, \widehat{\tt F}\cup\ttG\cup\ttH, \Psi^*\widehat\P)$, mutations for the quiver $\Psi^*\widehat Q$, and exchange relations for the almost-cluster variables in $\widehat {\tt F}=\{\hat \ttf_l\}$, thus giving rise to a (generalized) almost-cluster structure $\Psi^*\widehat{\CC}$ on $\AA^r$.
When $G=\varnothing$ we occasionally drop the hat notation and speak of 
$\Psi^*{\Sigma}=(\Psi^* Q, {\tt F}\cup\ttH, \Psi^*\P)$ with $\ttF=\{\ttf_l\}$, and when $\Psi$ is trivial, and hence 
$\ttH=\varnothing$, we occasionally drop the $\Psi$ notation and speak of  $\widehat{\Sigma}=(\widehat Q, \widehat{F}\cup G, \widehat\P)$.

We start from two similarly looking examples.

\begin{example} \label{goodex}
Assume that $s=4$ and $I=\varnothing$, so that $V=\AA^4$ with coordinates $y_1, y_2, y_3, y_4$. Define a cluster structure $\CC$ on $V$ 
by the initial cluster $\{f_1=y_1, f_2=y_2, f_3=y_3, f_4=y_1y_4-y_2y_3\}$ with frozen variables $f_2, f_3, f_4$ and a unique exchange relation $f_1f_1'=f_2f_3+f_4$, so that $f_1'=y_4$. Clearly, $\CC$ has only two distinct seeds and is regular. Further, consider $\AA^5$ with coordinates $x_1,\dots,x_5$ and a map $\Psi\:\AA^5\to\AA^4$ given by
$\Psi^*y_1=x_1/x_5$ and $\Psi^*y_i=x_i$ for $i=2,3,4$. Assume that $|J|=1$, and the only distinguished polynomial is 
$\tth=x_5$, so that $\ttf_i= x_i$ for $i=1,2,3$, $\ttf_1'=x_4$, $\ttf_4=  x_1x_4-x_2x_3x_5$ and 
$\lambda(f_1)=\lambda(f_4)=1$, $\lambda(f_2)=\lambda(f_3)=\lambda(f_1')=0$. Consequently, the exchange relation for 
$f_1$ 
can be rewritten as $(\ttf_1/\tth) \ttf_1'= \ttf_2 \ttf_3 + \ttf_4/\tth$, which yields  
$ \ttf_1 \ttf_1'= \ttf_2 \ttf_3 \tth+ \ttf_4$. The latter is a usual exchange relation for a cluster structure 
$\Psi^*{\CC}$ on $\AA^5$ with frozen variables $\ttf_2,  \ttf_3, \ttf_4$ and~$\tth$.
\end{example}

\begin{example} \label{badex}
In this example $V$ and $\CC$ are the same as in Example~\ref{goodex} and the map $\Psi\:\AA^5\to\AA^4$ is given by
$\Psi^*y_2=x_2/x_5$ and $\Psi^*y_i=x_i$ for $i=1,3,4$. The only distinguished polynomial is the same $\tth=x_5$, so that 
$\ttf_i= x_i$ for $i=1,2,3$, $\ttf_1'=x_4$, $\ttf_4=  x_1x_4x_5-x_2x_3$ and $\lambda(f_2)=\lambda(f_4)=1$, 
$\lambda(f_1)=\lambda(f_3)=\lambda(f_1')=0$. Consequently, the exchange relation for $f_1$ 
can be rewritten as $\ttf_1 \ttf_1'=(\ttf_2/\tth) \ttf_3 + \ttf_4/\tth$, which does not yield any cluster-like exchange relation for $\ttf_1$ and $\ttf_1'$.

One may try to fix this problem by considering $\bar \ttf_1'=\tth \ttf_1'$ that satisfies the cluster exchange relation 
$ \ttf_1\bar \ttf_1'= \ttf_2\ttf_3 +\ttf_4$. However, this construction is not canonical: one can consider instead 
$\bar \ttf_1=\tth\ttf_1$ that satisfies the cluster exchange relation $\bar \ttf_1 \ttf_1'= \ttf_2 \ttf_3 + \ttf_4$. Note that not only the corresponding cluster structures on $\AA^5$ are different, but the corresponding upper cluster algebras are different as well. Indeed, in the first case we have 
\begin{equation*}
\begin{aligned}
\UU_1&=\C[x_1^{\pm1},x_2,x_3,x_5,x_1x_4x_5-x_2x_3]\cap\C[(x_4x_5)^{\pm1},x_2,x_3,x_5,x_1x_4x_5-x_2x_3]\\
&=\C[x_1^{\pm1},x_2,x_3,x_5,x_4x_5]\cap\C[(x_4x_5)^{\pm1},x_1, x_2,x_3,x_5]\\
&=\C[x_1,x_2,x_3,x_5,x_4x_5],
\end{aligned}
\end{equation*}
so that $x_4\notin\UU_1$, while in the second case a similar computation gives
\begin{equation*}
\begin{aligned}
\UU_2&=\C[(x_1x_5)^{\pm1},x_2,x_3,x_5,x_1x_4x_5-x_2x_3]\cap\C[x_4^{\pm1},x_2,x_3,x_5,x_1x_4x_5-x_2x_3]\\
&=\C[x_2,x_3,x_4,x_5,x_1x_5],
\end{aligned}
\end{equation*}
and hence $\UU_1\ne\UU_2$.
 \end{example}

Example~\ref{badex} illustrates a general problem. By definition, an exchange relation for an arbitrary pair of adjacent seeds $\Sigma, \Sigma'$ in $\CC$ reads
$f_m^\Sigma f_m^{\Sigma'}=P_m^{\Sigma,\Sigma'}$ with $P_m^{\Sigma,\Sigma'}\in \C[\{f_l^\Sigma\}_{l\ne m},\{g_i^{\pm1}\}_{i\in I}]$,  where $f_m^\Sigma$ and $f_m^{\Sigma'}$ are cluster variables in $\Sigma$ and $\Sigma'$, respectivelly, and $P_m^{\Sigma,\Sigma'}$ is the right hand side of the corresponding exchange relation. This relation yields  
\begin{equation}\label{liftedexch}
\frac{\hat \ttf_m^\Sigma\hat\ttf_m^{\Sigma'}}{\prod_{k\in K}\ttq_k^{\lambda_k(f_m^\Sigma)+\lambda_k(f_m^{\Sigma'})}}
=\frac{\tilde {\tt P}_m^{\Sigma,\Sigma'}}{\prod_{k\in K}\ttq_k^{\Lambda_k(P_m^{\Sigma,\Sigma'})}}
\end{equation}
with $\tilde{\tt P}_m^{\Sigma,\Sigma'}\in\C\left[\{\hat\ttf_l^\Sigma\}_{l\ne m},\{\ttg_i\}_{i\in I}, \{\tth_j\}_{j\in J} \right]$ not divisible by any of $\ttq_k$. 
We say that there exists a {\it coherent\/} pullback of $\CC$ if for any seed $\Sigma$ and any mutable index $m$ 
the Laurent monomial
\begin{equation}\label{M}
M=\prod_{k\in K}\ttq_k^{\lambda_k(f_m^\Sigma)+\lambda_k(f_m^{\Sigma'})-\Lambda_k(P_m^{\Sigma,\Sigma'})}
\end{equation}
equals~1 and the family $\widehat\ttF^\Sigma\cup\ttG\cup\ttH$ is algebraically independent on $\AA^r$.

 As is shown in the next example, the notion of coherency is rather subtle: it might happen that $M=1$ for all exchange relations in some seed $\Sigma$, and $M\ne 1$ for an exchange relation in an adjacent seed $\Sigma'$.

\begin{example} \label{badex2}
Assume that $s=3$ and $I=\varnothing$, so that $V=\AA^3$ with coordinates $y_1, y_2, y_3$. 
Define a cluster structure $\CC$ on $V$ 
by the initial cluster $\{f_1=y_1, f_2=y_1^2+y_1y_2-1, f_3=y_3f_2^2-y_1-y_2\}$ with the frozen variable $f_3$ and  exchange relations $f_1f_1'=1+f_2$ and $f_2f_2'=1+f_1f_3$, so that $f_1'=y_1+y_2$ and $f_2'=y_1y_3f_2-1$. Since $\CC$
is of rank~2 with one frozen variable, it has exactly six distinct cluster variables: $f_1, f_2, f_3, f_1', f_2'$ listed above and one more variable 
$f_2''$ that satisfies the exchange relation $f_2f_2''=f_1'+f_3$ in the cluster adjacent to the initial one in 
direction~1, so that $f_2''=y_3f_2$. We thus see that $\CC$ is a regular cluster structure on $V$.

Further, consider $\AA^4$ with coordinates $x_1,\dots,x_4$ and a map $\Psi\:\AA^4\to\AA^3$ given by
$\Psi^*y_i=1+x_ix_4$ for $i=1,2$ and $\Psi^*y_3=x_3x_4$. Assume that $|J|=1$, and the only distinguished polynomial is 
$\tth=x_4$, so that $\Psi^*f_i=\ttf_i$ for $i=1,2,3$, $\Psi^*f_i'=\ttf_i'$ for $i=1,2$, and 
$\Psi^*f_2''=\ttf_2''\tth$. The right hand sides of exchange relations in the initial cluster are $P_1=1+f_2$ and 
$P_2=1+f_1f_3$, so that $\tilde{\tt P}_1=1+\ttf_2$ and $\tilde{\tt P}_2=1+\ttf_1\ttf_3$. Consequently, $M$
defined via~\eqref{M} equals~1. However, in the cluster adjacent to the initial one in 
direction~1, the right hand side of the new exchange relation for $f_2$ is $P_2''=f_1'+f_3$, so that 
$\tilde{\tt P}''_2=\ttf_1'+\ttf_3$, and hence~\eqref{M} yields 
$M=1/\tth$, and so coherence is violated.
\end{example}

The last example in this section exhibits the situation when $I\ne\varnothing$ and the cluster structure on $V$ is generalized. This example is important for the study of (generalized) cluster structures on $GL_n$ compatible with R-matrix Poisson brackets. Such brackets are classified by pairs of 
Belavin--Drinfeld triples $\bfG= (\Gamma_1,\Gamma_2,\gamma)$, where $\Gamma_1$ and $\Gamma_2$ 
are subsets of the set of positive simple roots in the $A_{n-1}$ root system and $\gamma$ 
is a nilpotent isometry (see~\cite{GSVple} for details). In~\cite{GSVple} we treated the subclass
of Belavin--Drinfeld triples that we called aperiodic, which leads to ordinary cluster structures on $GL_n$. 
The first instance of a periodic
Belavin--Drinfeld triple is treated in~\cite[Section~6]{GSVpest} as an application of the theory of periodic staircase matrices. It gives a generalized cluster structure on $GL_6$. The example below presents a similar construction in $GL_3$.

\begin{example} \label{cgacg3}
 Let $V=GL_3$ with coordinates $X=(x_{ij})_{i,j=1}^3$ given by the condition $\det X\ne 0$ in 
$\Mat_3=\AA^9$, so that $|I|=1$. The pair of Belavin--Drinfeld triples $\bfGr,\bfGc$ is given by 
\[
\Gamma_1^{\rm r}=\{\alpha_1\},\  \Gamma_2^{\rm r}=\{\alpha_2\}, \ \gammar(\alpha_1)=\alpha_2;\qquad
\Gamma_1^{\rm c}=\{\alpha_2\},\ \Gamma_2^{\rm c}=\{\alpha_1\}, \ \gammac(\alpha_2)=\alpha_1.
\]
The construction of~\cite{GSVple} produces the periodic matrix
\[
\nar\begin{bmatrix}
\dots &\dots  &\dots  &\dots  &\dots &\dots &\dots &\dots &\dots &\dots &\dots &\dots\\
\dots & x_{31}& x_{32}& x_{33}& 0    & 0    &0     & 0    & 0    & 0    & 0    & \dots \\
\dots & 0     & x_{11}& x_{12}& x_{13}& 0   & 0    &0     & 0    & 0    & 0    & \dots   \cr
\dots & 0     & x_{21}& x_{22}& x_{23}&x_{11}&x_{12}&x_{13} & 0  & 0    & 0    & \dots   \cr
\dots & 0     & x_{31}& x_{32}& x_{33}&x_{21}&x_{22}&x_{23} & 0  & 0    & 0    & \dots   \cr
\dots & 0     & 0     & 0     & 0     &x_{31}&x_{32}&x_{33} & 0  & 0    & 0    & \dots   \cr
\dots & 0     & 0     & 0     & 0     & 0    &x_{11}&x_{12} &x_{13}& 0  & 0    & \dots   \cr
\dots & 0     & 0     & 0     & 0     & 0    &x_{21}&x_{22} &x_{23}&x_{11}&x_{12}& \dots   \cr
\dots & 0     & 0     & 0     & 0     & 0    &x_{31}&x_{32} &x_{33}&x_{21}&x_{22}& \dots   \cr
\dots & 0     & 0     & 0     & 0     & 0    &0     &0      &0     &x_{31}&x_{32}& \dots   \cr
\dots &\dots  &\dots  &\dots  &\dots  &\dots &\dots &\dots  &\dots &\dots &\dots &\dots
\end{bmatrix}
\]
The corresponding characteristic polynomial is 
\[
\det\begin{bmatrix}
x_{13} & 0     & tx_{11}      & tx_{12}       \\
x_{23} & x_{11}& x_{12}+tx_{21}&x_{13}+tx_{22}\\
x_{33} & x_{21}& x_{22}+tx_{31}&x_{23}+tx_{32}\\
0      & x_{31}& x_{32}        &x_{33}
\end{bmatrix}=x_{13}\det X-tc(X)+t^2x_{31}\det X,
\]
and the core (see~\cite{GSVpest} for definitions) is a $5\times 5$ matrix
\[
\Phi=\begin{bmatrix} 
x_{23} & x_{11} & x_{12} & x_{13} & 0     \\
x_{33} & x_{21} & x_{22} & x_{23} & 0     \\
0      & x_{31} & x_{32} & x_{33} & 0     \\
0      & 0      & x_{21} & x_{22} & x_{23}\\
0      & 0      & x_{31} & x_{32} & x_{33}
\end{bmatrix}.
\]
The initial cluster $F_3=\{f_l\}_{l=1}^9$ consists of five mutable variables $f_i=\det\Phi_{[i,5]}^{[i,5]}$, $i\in[1,5]$, and four frozen variables $f_6=x_{31}$, $f_7=x_{13}$, $f_8=\det X$, $f_9=c(X)$. The initial quiver $Q_3$ is shown in 
Fig.~\ref{cgacgpestq}a); each variable $f_i$, $i\in[1,8]$, is assigned to vertex~$i$. Vertex~1 is special and has multiplicity~2. Vertex~9 is isolated and is not shown in Fig.~\ref{cgacgpestq}; the corresponding function $f_9$ enters only the middle coefficient in the generalized exchange relation at vertex~1, which reads
$f_1f_1'=f_2^2f_7+f_9f_2f_5+f_8^2f_5^2f_6$. Note that the middle coefficient in the latter exchange relation is not 
$f_9$, but rather $f_9/f_8$, so that the correct way to write  it is $f_1f_1'=f_2^2f_7+(f_9/f_8)f_2f_8f_5+f_8^2f_5^2f_6$. Using~\cite[Theorem~3.2]{GSVpest} together with the standard Desnanot--Jacobi identity and short 
Pl\"ucker relations one can prove that the generalized cluster structure 
$\GCC_3$ on $GL_3$ defined by the cluster $F_3$, quiver
$Q_3$, and string $(1, f_9/f_8,1)$ is regular.  Moreover, a direct check with Maple shows that the 
corresponding upper cluster algebra over $\C[f_6,f_7,f_9, f_8^{\pm1}]$ coincides with $\C[GL_3]$, and hence
$\GCC_3$ is complete.

\begin{figure}[ht]
\begin{center}
\includegraphics[width=8cm]{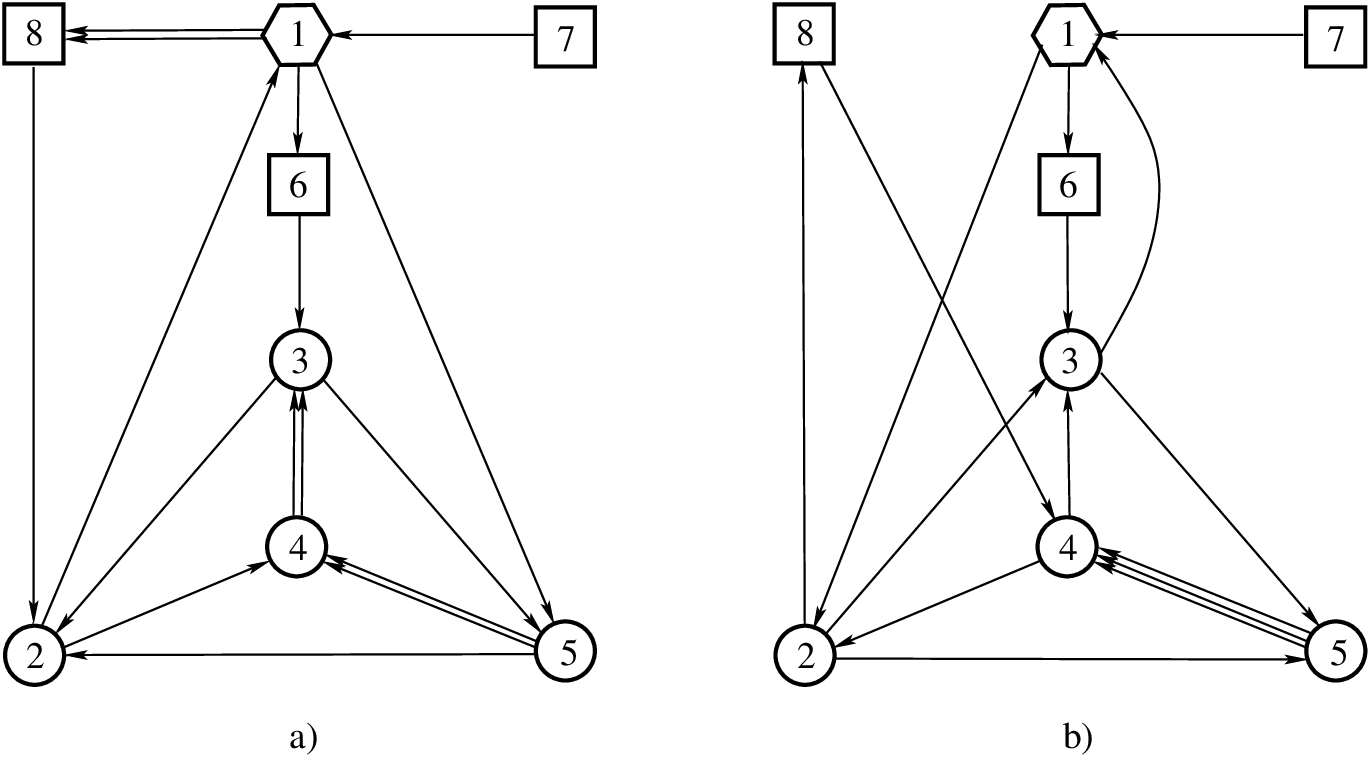}
\caption{Quiver $Q_3$: a) initial quiver; b) quiver after mutation at~2}
\label{cgacgpestq}
\end{center}
\end{figure}

Our goal is to lift $\GCC_3$ to the whole $\Mat_3$, hence $f_8=\det X$ becomes the only distinguished polynomial. For all other variables in the initial cluster $F_3$ one has $\hat f_i=f_i$ and $\hat f_i'=f_i'$, and the exchange relations are lifted to $\Mat_3$ in a trivial way. However, after the mutation at 
vertex~2 and replacing $f_2$ with $f_2'=x_{13}x_{32}x_{33}+x_{22}x_{23}x_{33}-x_{12}x_{33}^2-x_{23}^2x_{32}$
we get the quiver shown in Fig.~\ref{cgacgpestq}b). The new exchange relation for function 
$f_1$ reads $f_1f_1''=(f_3)^2+(f_9/f_8)f_3f_2'+(f_2')^2f_6f_7$, 
so that 
$\hat f_1''=f_1''f_8$, and hence the exchange relation for $\hat f_1$ reads 
\[
\hat f_1\hat f_1''=(\hat f_3)^2f_8+\hat f_9\hat f_3\hat f_2'+(\hat f_2')^2\hat f_6\hat f_7f_8.
\]
Note that the first and the last term in the right hand side of this relation are not coprime, as required in the definition of a generalized exchange relation. We are interested in recovering such relations for all clusters 
in $\GCC_3$. 
\end{example}

\section{Sufficient condition for the existence of a coherent pullback}\label{maincond}

In this section we obtain the following sufficient condition for the existence of a coherent pullback. Recall that 
a family of functions $\{\varphi_m\}_{m=1}^p$ is said to be algebraically dependent on a subvariety $W\subset \AA^r$ if there exists a polynomial $R(\varphi_1,\dots,\varphi_p)$ that vanishes identically on $W$, and algebraically independent if such a polynomial does not exist.

\begin{theorem}\label{maincondtheo}
A coherent pullback of a regular (generalized) cluster structure $\CC$ on $V$ with the initial cluster 
$F=\{f_l\}_{l=1}^s$ exists if for any $k\in K$ 
the family $(\hat\ttF\cup\ttG\cup\ttH)\setminus\{\ttq_{k}\}$
is algebraically independent on $\{\ttq_{k}=0\}\subset\AA^r$. 
\end{theorem}

\begin{proof} For a pair of adjacent seeds $\Sigma$, $\Sigma'$ and 
an exchange relation $f_m^\Sigma f_m^{\Sigma'}=P_m^{\Sigma,\Sigma'}$ in $\CC$, write down the evaluation of 
$\tilde\ttP_m^{\Sigma,\Sigma'}$ in the ring $\C[x_1,\dots,x_r]$ similarly to~\eqref{regpull} as a polynomial not divisible by any of $\ttq_k$ 
divided by $\prod_{k\in K}\ttq_k^{\lambda_k(\tilde \ttP_m^{\Sigma,\Sigma'})}$.
Clearly, $\lambda_k(\tilde \ttP_m^{\Sigma,\Sigma'})\le0$, 
since the evaluation is a polynomial. It follows immediately from~\eqref{liftedexch} that
\begin{equation}\label{equaldenom}
\lambda_k(\tilde \ttP_m^{\Sigma,\Sigma'})=\lambda_k(f_m^\Sigma)+\lambda_k(f_m^{\Sigma'})-\Lambda_k(P_m^{\Sigma,\Sigma'})
\end{equation}
so that condition $M=1$ for $M$ given by~\eqref{M} is equivalent to $\lambda_k(\tilde \ttP_m^{\Sigma,\Sigma'})=0$
for all $k\in K$. 

\begin{lemma}\label{onej}
Assume that for some $k\in K$ 
and some seed $\Sigma$, the family 
$(\hat\ttF^\Sigma\cup\ttG\cup\ttH)\setminus\{\ttq_k\})$
is algebraically independent on $\{\ttq_{k}=0\}\subset\AA^r$ and $\Sigma'$ is adjacent to $\Sigma$
in direction $m\in[1,s]$. Then $\lambda_{k}(\tilde \ttP_m^{\Sigma,\Sigma'})=0$.
\end{lemma}

\begin{proof} 
Indeed, condition $\lambda_{k}(\tilde \ttP_m^{\Sigma,\Sigma'})>0$ would mean that $\ttP_m^{\Sigma,\Sigma'}$ vanishes identically on $\{\ttq_{k}=0\}$. Note that $\ttP_m^{\Sigma,\Sigma'}$ is not divisible by $\ttq_{k}$ in the ring 
$\C\left[\{\hat\ttf_l^t\}_{l\ne m}, \{\tth_j\}_{j\in J} \right]$. Therefore, the polynomial
$\ttP_m^{\Sigma,\Sigma'}\bigr\rvert_{\ttq_{k}=0}$
certifies an algebraic dependence of the family $(\hat\ttF^\Sigma\cup\ttH)\setminus\{\ttq_k\})$
on $\{\ttq_{k}=0\}$, a contradiction.
\end{proof}
 
\begin{lemma} 
Assume that for some seed $\Sigma$ the families 
$(\hat\ttF^\Sigma\cup\ttG\cup\ttH)\setminus\{\ttq_{k}\})$
are algebraically independent on $\{\ttq_{k}=0\}\subset\AA^r$ for all $k\in K$.
Then the same holds for the families $(\hat\ttF^{\Sigma'}\cup\ttG\cup\ttH)\setminus\{\ttq_{k}\})$
for any $\Sigma'$ mutationally equivalent to $\Sigma$.
\end{lemma}

\begin{proof} It suffices to prove the assertion of the lemma for a seed $\Sigma'$ adjacent to $\Sigma$. 
Assume that $\Sigma'$ is adjacent to $\Sigma$ in  direction $m$.
It follows from Lemma~\ref{onej} 
that under the assumptions of the lemma,
 $\lambda_k(\tilde \ttP_m^{\Sigma,\Sigma'})=0$ 
for all $k\in K$. 
Consequently, conditions~\eqref{equaldenom} guarantee that~\eqref{liftedexch} is satisfied with $M=1$, so that 
$\hat\ttf_m^{\Sigma'}=\tilde\ttP_m^{\Sigma,\Sigma'}/\hat\ttf_m^\Sigma$. Therefore, an algebraic dependence for a family
$(\hat\ttF^{\Sigma'}\cup\ttH)\setminus\{\ttq_{k}\}$
on $\{\ttq_{k}=0\}$ would imply an algebraic dependence of the family 
$(\hat\ttF^\Sigma\cup\ttH)\setminus\{\ttq_{k}\}$
on the same subvariety, a contradiction.
\end{proof}

It remains to prove that the conditions of the theorem imply that the family $\widehat\ttF^\Sigma\cup\ttG\cup\ttH$ is algebraically independent on $\AA^r$. Indeed, assume the contrary, and let $R$ be an irreducible polynomial that certifies the algebraic dependence. Pick an arbitrary $\ttq_k$ and write $R=\ttq_kR_1+R_0$  where $R_0$ is not divisible by $\ttq_k$. Since $R$ is irreducible, $R_0$ is a non-trivial polynomial that vanishes identically on $\{\ttq_k=0\}$, a contradiction. 
\end{proof}

The following example shows that the assumptions of Theorem~\ref{maincondtheo} are not necessary for the existence of the regular pullback.

\begin{example}\label{regdepend}
Assume that $s=3$ and $I=\varnothing$, so that $V=\AA^3$ with coordinates $y_1, y_2, y_3$. 
Define a cluster structure $\CC$ of rank~1 on $V$ 
by the initial cluster $\{f_1=y_1, f_2=y_1^2+y_1y_2-y_3, f_3=y_3\}$ with frozen variables $f_2, f_3$ and  the 
exchange relation $f_1f_1'=f_2+f_3$. Clearly, $f_1'=y_1+y_2$, so that $\CC$ is a regular cluster structure on $V$.

Further, consider $\AA^4$ with coordinates $x_1,\dots,x_4$ and a map $\Psi\:\AA^4\to\AA^3$ given by
$\Psi^*y_1=x_1$ and $\Psi^*y_i=x_ix_4$ for $i=2,3$. Assume that $|J|=1$, and the only distinguished polynomial is 
$\tth=x_4$, so that $\Psi^*f_i=\ttf_i$ for $i=1,2$, $\Psi^*f_3=\ttf_3\tth$, and $\Psi^*f_1'=\ttf_1'$. 
The right hand side of the exchange relation in the initial cluster is $P=f_2+f_3$, so that 
$\tilde{\ttP}=\ttf_2+\tth\ttf_3$. Consequently, the evaluation of $\tilde\ttP$ in $\C[x_1,x_2,x_3,x_4]$ 
equals $x_1^2+x_1x_2x_4$, and hence the coherent pullback of $\CC$ to $\AA^4$ exists. On the other hand, 
the family $\{\ttf_1, \ttf_2, \ttf_3\}$ is algebraically dependent on the
subvariety $\{\tth=x_4=0\}$ since $\ttf_2\bigr\rvert_{\tth=0}=\ttf_1^2$.  
\end{example} 

\section{Combinatorics of a coherent pullback}\label{combin}
\subsection{The pullback of a seed}\label{seedpullback}
We assume that the generalized cluster structure $\CC$ on $V$ is described with the help of quiver with multiplicities
$Q=(Q,d_1,\dots,d_n)$, where $n$ is the number of mutable vertices, see~\cite{GSVdouble, GSVnewdouble} for details; without loss of generality we assume that vertices $1,\dots,n$ are mutable and vertices $n+1,\dots,s$ are frozen. Recall that the standard definition of the quiver mutation in direction $m$ for the quiver with multiplicities is modified as follows: if both vertices $i$ and $j$
in a path $i\to m\to j$ are mutable, then this path contributes $d_m$ edges $i\to j$ to the mutated quiver $Q'$; if exactly one of the vertices $i$ or $j$ is frozen then the path contributes $d_j$ or $d_i$ edges $i\to j$ to $Q'$, respectively. The multiplicities at the vertices do not change.  The adjacent cluster variable in direction 
$m$ is given by the generalized exchange relation $f_mf'_m=P_m=\sum_{r=0}^{d_m}T_r$ with
\begin{equation}\label{exchange}
T_r=p_{mr}\prod_{1\le l\le n} f_l^{rb_{ml}}\prod_{n<l\le s}f_l^{\lfloor rb_{ml}/d_m\rfloor}
\prod_{1\le l\le n}f_l^{(d_m-r)b_{lm}}\prod_{n<l\le s}f_l^{\lfloor (d_m-r)b_{lm}/d_m\rfloor};
\end{equation}
here $b_{ml}$ is the number of arrows from $m$ to $l$, $b_{lm}$ is the number of arrows from $l$ to $m$, and
$p_{mr}$ are exchange coefficients that satisfy condition $p_{m0}=p_{md_m}=1$ and are mutated via 
$p'_{mr}=p_{m,d_m-r}$.  The {\it $y$-variable\/} $t_m$ at vertex $m$ is 
\[
t_m=\prod_{1\le l\le s} f_l^{b_{ml}-b_{lm}},
\]
see~\cite[Section 6.1]{GSVdouble}.

We define the quiver $\Psi^*\widehat Q$ in a two-step procedure. First, take a copy of the quiver $Q$ and add a new frozen vertex for every $j\in J$ (note that vertices corresponding to $i\in I$ are already present in $Q$). The variables assigned to the vertices $k\in K$ are $\ttq_k$, and the variables assigned to other vertices are $\hat\ttf_l$, so that the set of variables is 
$\widehat\ttF\cup\ttG\cup\ttH$. All arrows between the mutable vertices of $Q$ are retained in 
$\Psi^*\widehat Q$. Additionally, for any $k\in K$, the {\it discrepancy\/}
\begin{equation}\label{newarrows}
\delta_{km}=\sum_{l=1}^n d_mb_{lm}\lambda_k(f_l)+\sum_{\substack {l=n+1\\ l\ne k}}^s b_{lm}\lambda_k(f_l)
-\sum_{l=1}^n d_mb_{ml}\lambda_k(f_l)  - \sum_{\substack {l=n+1\\ l\ne k}}^s b_{ml}\lambda_k(f_l)
\end{equation}
is computed and $|\delta_{km}|$ 
arrows are added between $m$ and $k$. These arrows are directed from $m$ to $k$ if $\delta_{km}$ in~\eqref{newarrows} is positive, and from $k$ to $m$ otherwise. As usual, for $k\in I$, the cycles of length~2 formed by the new arrows and existing arrows between $m$ and $k$ should be deleted.  Finally, the pullback of a string 
is defined as $\Psi^*(p_{mr})_{r=1}^{d_m} = (\hat \ttp_{mr})_{r=1}^{d_m}$, so that $\Psi^*\widehat\Sigma=(\Psi^*Q,
\widehat \ttF\cup\ttG\cup\ttH, \Psi^*\widehat\P)$ is a pullback of $\Sigma=(Q,F\cup G,\P)$. 
The following statement is immediate.

\begin{proposition}\label{ypullback}
The $y$-variables for the seed $\Psi^*\widehat\Sigma$ coincide with the pullback of $y$-variables for the seed $\Sigma$.
\end{proposition}

The quiver $\Psi^*\widehat Q$ defined above is enough for an adequate description of the regular pullback of a cluster structure, see Section~\ref{liftquiver}. For the case of generalized cluster structures we will further modify the definition of $\Psi^*\widehat Q$ in
Section~\ref{almostquiver} below.

\subsection{Mutations of a coherent pullback}
In what follows we assume that the coherency condition $M=1$ for $M$ defined in~\eqref{M} is satisfied for any seed in 
$\CC$. Consequently, 
\begin{equation}\label{cohcond}
\Lambda_k(P_m)=\max_{0\le r\le d_m}\lambda_k(T_r), \quad k\in K, 
\end{equation}
and the exchange relation~\eqref{liftedexch} for $\hat\ttf_m$ and $\hat\ttf_m'$ reads
\begin{equation}\label{pullex} 
\hat\ttf_m\hat\ttf_m'=\sum_{r=0}^{d_m} \hat\ttT_r\prod_{k\in K}\ttq_k^{\mu_{km}(r)}
\end{equation}
with $\mu_{km}(r)=\Lambda_k(P_m)-\lambda_k(T_r)$;
here and in what follows we suppress explicit mention of adjacent seeds $\Sigma$ and $\Sigma'$. Further,
by~\eqref{exchange}, 
\begin{align*}
\lambda_k(T_r)&=\lambda_k(p_{mr})+r\sum_{l=1}^n b_{ml}\lambda_k(f_l)+\sum_{l=n+1}^s 
\left\lfloor \frac{rb_{ml}}{d_m}\right\rfloor\lambda_k(f_l)\\
& +(d_m-r)\sum_{l=1}^n b_{lm}\lambda_k(f_l)+\sum_{l=n+1}^s 
\left\lfloor \frac{(d_m-r)b_{lm}}{d_m}\right\rfloor\lambda_k(f_l),\quad k\in K, r\in[0,d_m].
\end{align*}
Denoting $\{a\}=a-\lfloor a\rfloor$ for any number $a$, one can rewrite the latter expression as
$\lambda_k(T_r)=\chi_{km}(r)+\tau_{km}r+\sigma_{km}$ with
\begin{equation}\label{lamcoef}
\begin{aligned}
\chi_{km}(r)=&\lambda_k(p_{mr})-
\sum_{l=n+1}^s\left(\left\{\frac{rb_{ml}}{d_m}\right\}+\left\{\frac{(d_m-r)b_{lm}}{d_m}\right\}\right)\lambda_k(f_l),\\
\tau_{km}=&\sum_{l=1}^n(b_{ml}-b_{lm})\lambda_k(f_l)+\sum_{l=n+1}^s\frac{b_{ml}-b_{lm}}{d_m}\lambda_k(f_l),\\
\sigma_{km}=&d_m\sum_{l=1}^n b_{lm}\lambda_k(f_l)+\sum_{l=n+1}^s b_{lm}\lambda_k(f_l).
\end{aligned} 
\end{equation}
Taking into account that $\chi_{km}(0)=\chi_{km}(d_m)=0$, we arrive at
\begin{equation*}
\tau_{km}=\frac{\lambda_k(T_{d_m})-\lambda_k(T_0)}{d_m}, \quad \sigma_{km}=\lambda_k(T_0),
\end{equation*}
which finally yields
\begin{equation}\label{mutr}
\mu_{km}(r)=-\chi_{km}(r)+r\frac{\mu_{km}(d_m)-\mu_{km}(0)}{d_m}+\mu_{km}(0).
\end{equation}

Note that after the mutation at $m$, one has $b'_{ml}=b_{lm}$ and $b'_{lm}=b_{ml}$, so that
\[
\chi'_{km}(r)=\chi_{km}(d_m-r),\qquad \tau'_{km}=-\tau_{km},\qquad \sigma'_{km}=d_m\tau_{km}+\sigma_{km}.
\]
Let $\chi_{km}^\Sigma(r)$ denote the analog of $\chi_{km}(r)$ written for vertex $m$ in an arbitrary seed $\Sigma$.

\begin{proposition}\label{samefy}
For any $\Sigma$, either $\chi_{km}^\Sigma(r)=\chi_{km}(r)$, or $\chi_{km}^\Sigma(r)=\chi'_{km}(r)$, depending on the parity of the number of mutations at $m$ in a sequence connecting $\Sigma$ with the initial seed.
\end{proposition}

\begin{proof} The statement is clear for the first term in the expression for $\chi_{km}(r)$. Let us prove that the second term does not change after the mutation at $i\ne m$. First, all vertices $l\in[n+1,s]$ are frozen, so 
$\lambda_k(f_l)$ does not change during mutations. Note that at least one of $b_{ml}$ and $b_{lm}$ vanishes. Without loss of generality we assume that $b_{lm}=0$ and $b_{ml}\ge0$. If there is no path $l\to i\to m$ or $m\to i\to l$ in the quiver, mutation at $i$ does not influence $b_{ml}$ and $b_{lm}$. Otherwise, either $b_{ml}'=b_{ml}\pm d_m\ge0$ and $b'_{lm}=0$, or 
$b'_{ml}=0$ and $b'_{lm}=d_m-b_{ml}>0$. In the first case,
\[
\left\{\frac{rb'_{ml}}{d_m}\right\}=\left\{\frac{rb_{ml}}{d_m}\pm r\right\}=\left\{\frac{rb_{ml}}{d_m}\right\},\quad
\left\{\frac{(d_m-r)b'_{lm}}{d_m}\right\}=\left\{\frac{(d_m-r)b_{lm}}{d_m}\right\}=0,
\]
while in the second case,
\[
\left\{\frac{rb'_{ml}}{d_m}\right\}=0,\quad \left\{\frac{(d_m-r)b'_{lm}}{d_m}\right\}=
\left\{\frac{(d_m-r)(d_m-b_{ml})}{d_m}\right\}=\left\{\frac{rb_{ml}}{d_m}\right\},
\]
hence in both cases the coefficient at $\lambda_k(f_l)$ in~\eqref{lamcoef} is preserved.
\end{proof}

It follows from~\eqref{mutr} and Proposition~\ref{samefy} that behavior of $\mu_{km}(r)$ is defined by the behavior
of $\mu_{km}(0)$ and $\mu_{km}(d_m)$. Denote by $\mu_{km}^i(0)$ and $\mu_{km}^i(d_m)$ the values of $\mu_{km}(0)$ and 
$\mu_{km}(d_m)$ after mutation at $i$. Clearly, if there are no arrows between $i$ and $m$ then 
$\mu_{km}^i(0)=\mu_{km}(0)$ and $\mu_{km}^i(d_m)=\mu_{km}(d_m)$. Otherwise, $\mu_{km}^i(0)$ and $\mu_{km}^i(d_m)$ are
obtained as follows.

\begin{proposition}\label{newmu}
{\rm (i)} If $b_{mi}>0$ then
\begin{equation}\label{newmufrom}
\begin{gathered}
\mu_{km}^i(0)=\max_{0\le r\le d_m}\left\{\chi_{km}(r)-\frac r{d_m}\left(\mu_{km}(d_m)-\mu_{km}(0)\right)
+rb_{mi}\mu_{ki}(0)\right\},\\
\mu_{km}^i(d_m)-\mu_{km}^i(0)=\mu_{km}(d_m)-\mu_{km}(0)-d_mb_{mi}\mu_{ki}(0).
\end{gathered}
\end{equation}
{\rm (ii)} If $b_{im}>0$ then
\begin{equation}\label{newmutom}
\begin{gathered}
\mu_{km}^i(d_m)=\max_{0\le r\le d_m}\left\{\chi_{km}(d_m-r)-\frac r{d_m}\left(\mu_{km}(0)-\mu_{km}(d_m)\right)
+rb_{im}\mu_{ki}(d_i)\right\},\\
\mu_{km}^i(0)-\mu_{km}^i(d_m)=\mu_{km}(0)-\mu_{km}(d_m)-d_mb_{im}\mu_{ki}(d_i).
\end{gathered}
\end{equation}
\end{proposition}

\begin{proof}
(i) Recall that $\Lambda_k(P_m)=\mu_{km}(0)+\lambda_k(T_0)=\mu_{km}(d_m)+\lambda_k(T_{d_m})$. Equating the two expressions above we get
\begin{equation}\label{before}
\mu_{km}(0)
+d_mb_{mi}\lambda_k(f_i)+\lambda_k^{\mathrm {out}}
=\mu_{km}(d_m)
+\lambda_k^{\mathrm {in}}
\end{equation}
with
\begin{align*}
\lambda_k^{\mathrm {out}}&=d_m\sum_{\substack{l=1\\l\ne i}}^nb_{ml}\lambda_k(f_l)+\sum_{l=n+1}^sb_{ml}\lambda_k(f_l),\\
\lambda_k^{\mathrm {in}}&=d_m\sum_{l=1}^nb_{lm}\lambda_k(f_l)+\sum_{l=n+1}^sb_{lm}\lambda_k(f_l).
\end{align*}
In a similar way, after mutation at $i$ we have
\begin{multline}\label{after}
\mu_{km}^i(0)
+\lambda_k^{\mathrm {out}}+d_m\sum_{l=1}^nb_{mi}d_ib_{il}\lambda_k(f_l)
+\sum_{l=n+1}^sb_{mi}d_mb_{il}\lambda_k(f_l)\\
=\mu_{km}^i(d_m)
+\lambda_k^{\mathrm {in}}+d_mb_{mi}\lambda_k(f'_i).
\end{multline}
Subtracting~\eqref{after} from~\eqref{before}, 
we arrive at
\begin{multline*}
\mu_{km}^i(d_m)-\mu_{km}^i(0)=\mu_{km}(d_m)-\mu_{km}(0)\\
+d_mb_{mi}\left(d_i\sum_{l=1}^nb_{il}\lambda_k(f_l)
+\sum_{l=n+1}^sb_{il}\lambda_k(f_l)-\lambda_k(f_i)-\lambda_k(f'_i)\right).
\end{multline*}
To get the second relation in~\eqref{newmufrom}, it remains to notice that the expression in the brackets above equals
$-\mu_{ki}(0)$.

To get the first relation in~\eqref{newmufrom}, we write down~\eqref{mutr} for $\mu_{km}^i(r)$ 
taking into account that $\chi^i_{km}(r)=\chi_{km}(r)$ by Proposition~\ref{samefy},
and use condition
\[\min_{0\le r\le d_m}\mu_{km}^i(r)=0,\] 
which is equivalent to~\eqref{cohcond}. Consequently,
\[
\mu_{km}^i(0)=\max_{0\le r\le d_m}\left\{\chi_{km}(r)-\frac r{d_m}\left(\mu_{km}^i(d_m)-\mu_{km}^i(0)\right)\right\},
\]
and the first relation in~\eqref{newmufrom} follows from the second one.

(ii) The proof is similar, with~\eqref{mutr} replaced by
\[
\mu_{km}(d_m-r)=-\chi_{km}(d_m-r)+\frac r{d_m}(\mu_{km}(0)-\mu_{km}(d_m))+\mu_{km}(d_m).
\]
\end{proof}

\subsection{The simplest case: $d_m=1$}\label{liftquiver}
Mutation rules given in Proposition~\ref{newmu} are especially simple when $d_m=1$ for all $m\in[1,n]$, that is,
when $\CC$ is a regular cluster structure on $V$. In this case at least one of $\mu_{km}(0)$ and $\mu_{km}(1)$ vanishes,
and mutation of $\Psi^*\widehat Q$ is being carried out according to the standard rules of quiver mutation. For $r=s$ and 
$\Psi=\Id$ the cluster structures $\CC$ and $\widehat \CC$ are quasi-isomorphic, see~\cite{quasichris}.

\begin{example}\label{gl3st}
Let $V$ be a quasi-affine subvariety in $\Mat_3=\AA^9$ with coordinates $Y=(y_{ij})_{i,j=1}^3$ given by the condition
$g=\det Y_{12}^{23}\ne 0$. The initial cluster $F=\{f_l\}_{l=1}^9$ consists of four mutable variables $f_1=y_{33}$, 
$f_2=y_{32}$, $f_3=y_{23}/g$, $f_4=\det Y_{23}^{23}/g$, and five frozen variable $f_5=y_{13}$,  $f_6=\det Y$,
$f_7=\det Y_{23}^{12}$, $f_8=y_{31}$,  $f_9=g$. The initial quiver $Q$ is shown in Fig.~\ref{gl3qu}a).
 It is easy to check, using~\cite[Theorem~2.10]{BFZ}, that the obtained cluster structure $\CC$ on $V$ is regular.

\begin{figure}[ht]
\begin{center}
\includegraphics[width=12cm]{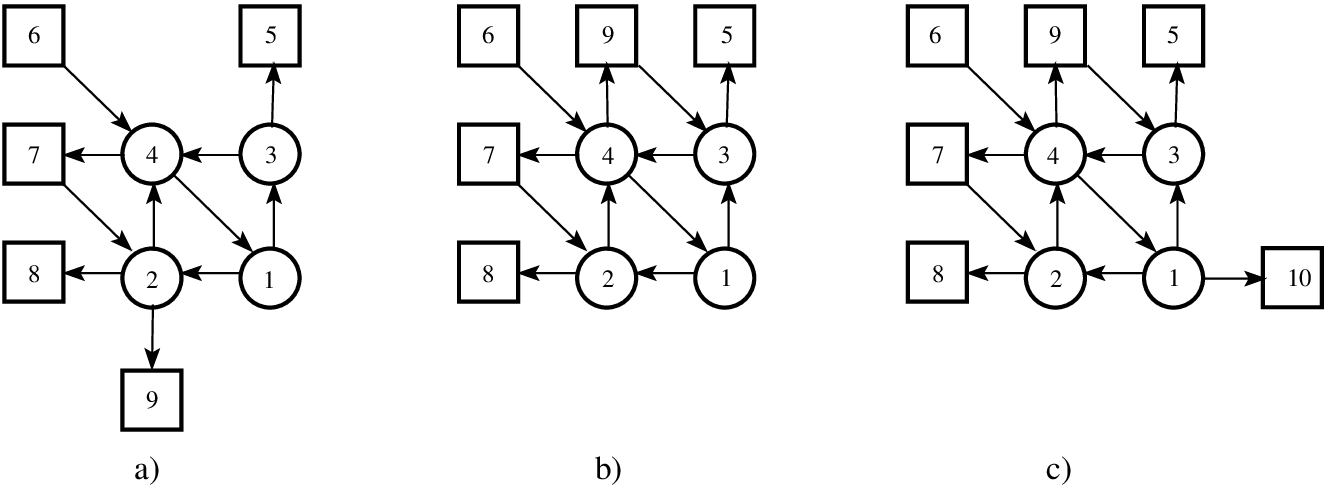}
\caption{Quiver and its pullbacks: a) initial quiver $Q$; b) quiver $\widehat Q$; c) quiver $\Psi^*\widehat Q$}
\label{gl3qu}
\end{center}
\end{figure}
 
Assume that $g$ is the distinguished polynomial.
It is easy to check that the condition of Theorem~\ref{maincondtheo} is satisfied, so there exists a regular pullback of 
$\CC$ to $\Mat_3$ with $\hat f_l=f_l$ for $l\ne 3,4$, $\hat f_3=y_{23}$, $\hat f_4=Y_{23}^{23}$.
Further, $\lambda_g(f_3)=\lambda_g(f_4)=1$,  and $\lambda_g(f_l)=0$ for $l\ne 3,4,9$, so that the discrepancy computed 
via~\eqref{newarrows} vanishes for
$m=1$, equals~$-1$ for $m=2,3$, and equals~1 for $m=4$. Consequently,  arrows from~9 
to~2, from~9 to~3, and from~4 to~9 are added; the first of them cancels the existing arrow from~2 to~9. The corresponding quiver $\widehat Q$ is shown in Fig.~\ref{gl3qu}b). This gives an example of a quasi-isomorphism between two cluster structures.

Further, consider $\AA^{10}$ with coordinates $X=(x_{ij})_{i,j=1}^3$ and $z$ a map $\Psi:\AA^{10}\to\AA^9$ given by
$\Psi^*y_{ij}=x_{ij}$ for $(i,j)\ne (3,3)$ and $\Psi^*y_{33}=x_{33}/z$. Assume that $|J|=1$ and the only new distinguished polynomials is $\tth=z$. Once again, it is easy to check that the condition of Theorem~\ref{maincondtheo} is satisfied, so there exists a regular pullback of $\widehat\CC$ to $\AA^{10}$. Further, $\lambda_\tth(\hat f_l)=1$ for $l=1, 4, 6$  and $\lambda_\tth(\hat f_l)=0$ otherwise.  We thus get that the discrepancy equals~1 for
$m=1$ and vanishes for for $m=2,3,4$, so that  an arrow from~1 
to~10 is added. The corresponding quiver $\Psi^*\widehat Q$ is shown in  Fig.~\ref{gl3qu}c). Note that the same quiver with the arrow from~1 to~10 reversed and $z$ attached to vertex~10 replaced by $1/z$ is the seed of the standard cluster structure on the Grassmannian $\Gr(3,6)$ corresponding to the chart
\[
\begin{bmatrix}
x_{11} & x_{12} & zx_{13} & 0 & 0 & 1\\
x_{21} & x_{22} & zx_{23} & 0 & -1 & 0\\
x_{31} & x_{32} & x_{33} & 1/z & 0 & 0
\end{bmatrix}.
\]
\end{example}

\subsection{The case $d_m>1$}\label{almostquiver}
For generalized cluster structures, the vertices of $\Psi^*\widehat Q$ are exactly the same as above, and so are 
the arrows at mutable vertices of multiplicity~1. At vertices of higher multiplicity 
both $\mu_{km}(0)$ and $\mu_{km}(d_m)$ may simultaneously differ from zero.
In this situation the quiver $\Psi^*\widehat Q$ has $\mu_{km}(0)$ arrows from $k$ to $m$ and  
$\mu_{km}(d_m)$ arrows from $m$ to $k$. Second relations in~\eqref{newmufrom} and~\eqref{newmutom} show that that the difference between the number of arrows in the opposite direction behaves under mutations exactly as the number of arrows for a quiver corresponding to a generalized cluster structure, and Proposition~\ref{ypullback} remains true.
However, the actual number of arrows in each direction depends on the value of $r$ at which the maximum in the first relations 
in~\eqref{newmufrom} and~\eqref{newmutom} is attained.

The situation depends on the function $\chi_{km}(r)$, or more exactly, on the upper convex hull
$C_{km}$ of the points  $(r,\chi_{km}(r))$, $r\in[0,d_m]$. The values $\mu_{km}(0)$ and $\mu_{km}(d_m)$ are the ordinates of the intersection points of support lines to $C_{km}$ with a slope $c-\tau_{km}$, $c\in\Z$, with the lines 
$r=0$ and $r=d_m$. If every support line 
of this type passes either through $(0,\chi_{km}(0))$ or through $(d_m,\chi_{km}(d_m))$ then, similarly to the case of ordinary cluster structures, at least one of $\mu_{km}(0)$ and $\mu_{km}(d_m)$ vanishes, and the quiver 
$\Psi^*\widehat Q$ is mutated according to the usual rules.

The situation becomes more complicated when there exist support lines to $C_{km}$ with slopes as indicated above that do not pass through $(0,\chi_{km}(0))$ or $(d_m,\chi_{km}(d_m))$, in which case both $\mu_{km}(0)$ and $\mu_{km}(d_m)$ simultaneously differ from zero. 
The simplest such situation occurs when there is exactly one such line, as is shown in the next example.

\begin{figure}[ht]
\begin{center}
\includegraphics[width=7.6cm]{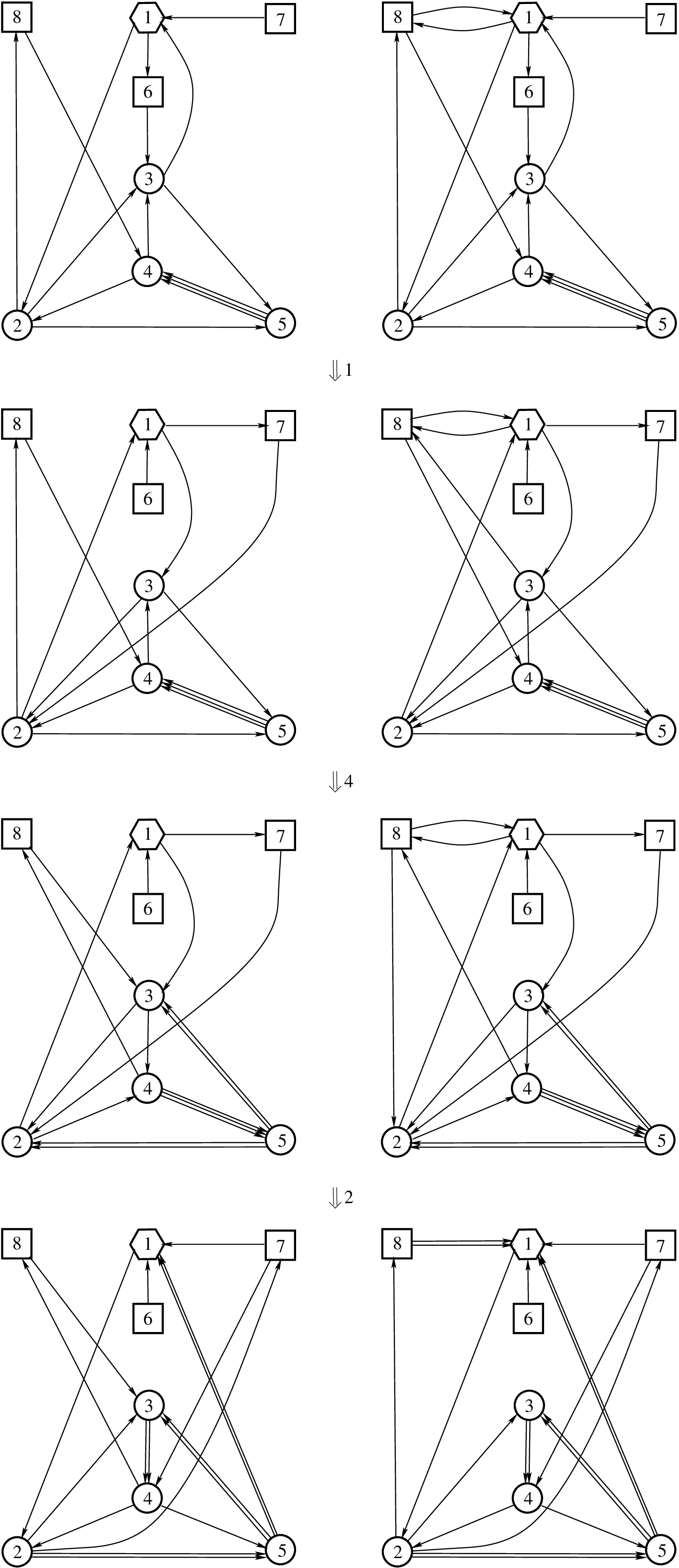}
\caption{Mutations of $Q_3$ (left) and $\widehat{Q}_3$ (right)}
\label{cgacgpestqq}
\end{center}
\end{figure}

\begin{example}\label{cgacg3mat}
We want to lift the generalized cluster structure $\GCC_3$ on $GL_3$  described in Example~\ref{cgacg3} to the whole 
$\Mat_3$. The only generalized vertex is vertex~1; it is easy to see that $C=C_{81}$ is the triangle with vertices 
$(0,0)$, $(1,1)$, and $(2,0)$, and that $\tau=\tau_{81}=1$ for the initial seed $(Q_3,F_3)$, which corresponds to the
support line with the slope~$-1$ that passes through $(2,0)$. Further, the discrepancy vanishes at all
mutable vertices, and hence the quiver $\widehat Q_3$ for the initial seed does not differ from $Q_3$ shown in Fig.~\ref{cgacgpestq}a).  The only support line to 
$C$ with an integer slope that does not pass through $(0,0)$ or $(2,0)$ is the horizontal line, in which case
$\mu_{81}(0)=\mu_{81}(2)=1$. This is exactly what happens after mutation at~2, and hence the quiver now  
looks different than the one shown in Fig.~\ref{cgacgpestq}b): it has an arrow from~1 to~8 and arrow from~8 to~1, 
see the upper right corner in Fig.~\ref{cgacgpestqq}. Let us now proceed with mutations at~1,~4, and~2. The corresponding quivers are shown in the left part of Fig.~\ref{cgacgpestqq}. The right part of the same figure shows
the quivers obtained from the one shown in the upper right corner via the same sequence of mutations. Note
the following differences between the two sets of quivers. After the first mutation, the quiver on the right keeps arrows $8\to1$ and $1\to8$ (since $\tau$ is still equal to~0), the arrow $3\to8$ is added due to the path $3\to1\to8$, and the arrow $2\to8$ disappears due to the path $8\to1\to2$. After the second mutation, 
the quiver on the right still keeps arrows $8\to1$ and $1\to8$, the arrow $8\to2$ is added due to the path $8\to4\to2$, and the arrow $3\to8$ disappears due to the path $8\to1\to3$. After the third mutation, the quiver on the right has an additional arrow $2\to8$ obtained by reversal of $8\to2$ and two additional arrows  $8\to1$ due to the path $8\to2\to1$ (note that now $\tau=-1$ and the support line to $C$ passes through $(0,0)$). At each step the quiver on the right can be obtained from the corresponding quiver on the left by the rules described above. For example, for the quiver in the last row, $\lambda_8(f_1)=\lambda_8(f_2)=1$ and $\lambda_8(f_l)=0$ for $l\ne 1,2$, so that the discrepancy
vanishes at vertex~5 and equals~1 at vertices~2 and~3. Consequently,  
 the arrow $2\to8$ is added in the quiver on the right, and the arrow $8\to3$ is cancelled.
Further, the discrepancy equals~$-1$ at vertex~4, hence the arrow $4\to8$ is cancelled in the quiver on the right. Finally, at vertex~1 the discrepancy equals~$-2$, hence two arrows $8\to1$ are added.
\end{example}

\begin{remark}
The situation when $\tau_{km}$ defined via the second equation in~\eqref{lamcoef} is an integer and all support lines with integer slopes pass either through $(0,0)$ or through $(d_m,0)$ except for the horizontal support line at height $h$ can be summarized as follows:

(i) compute the new value of $\Delta_{km}=\mu_{km}(0)-\mu_{km}(d_m)$ via the second formula in~\eqref{newmufrom} or in~\eqref{newmutom}, depending on the direction of arrows between $m$ and $i$;

(ii) if $\Delta_{km}=0$ then draw $h$ arrows from $m$ to $k$ and $h$ arrows from $k$ to $m$; if $\Delta_{km}>0$ then 
draw $\Delta_{km}$ arrows from $k$ to $m$; if $\Delta_{km}<0$ then 
draw $|\Delta_{km}|$ arrows from $m$ to $k$. 

Such a situation arises, for example, when $\lambda_{k}(p_{mr})=1$ for $r\in[1,d_m-1]$ and $\lambda_k(f_l)=0$ for 
$l\in [n+1,s]$, and hence $h=1$.
\end{remark}


The next example shows various situations that may arise in the process of regular pullback from $\AA^s$ to $\AA^r$.

\begin{example}
Assume that a generalized exchange relation in $\CC$ reads $ff'=f_1^d +\sum_{r=1}^{d-1} p_rf_1^{d-r}f_2^r+f_2^df_3$,
where $f_1$, $f_2$ are mutable and $f_3$ is frozen. 

(i) Assume additionally that there is one distinguished polynomial $\tth$ and $\lambda_{\tth}(p_r)=1$, $r\in[1,d-1]$, 
$\lambda_\tth(f_1)=\lambda_\tth(f_2)=0$, $\lambda_\tth(f_3)=1$. Then $\chi_\tth(r)$ computed via the first relation in~\eqref{lamcoef} is given via 
$\chi_\tth(r)=1-r/d$ for $r>0$, $\chi_\tth(0)=0$, and $\tau_\tth$ computed via the second relation in~\eqref{lamcoef} equals $1/d$. The corresponding upper convex hull $C_\tth$ for $d=6$ is shown in 
Fig.~\ref{convhulls}a). Support lines with the slopes $c-1/6$, $c\in\Z$, pass either through $A=(0,0)$ or through $B=(6,0)$, hence 
the arrows between the two vertices go in the same direction: either $6c+1$, $c\ge0$, arrows from the frozen vertex $k$ corresponding to $\tth$, or $-6c-1$, $c<0$, arrows to $k$.

\begin{figure}[ht]
\begin{center}
\includegraphics[width=12.5cm]{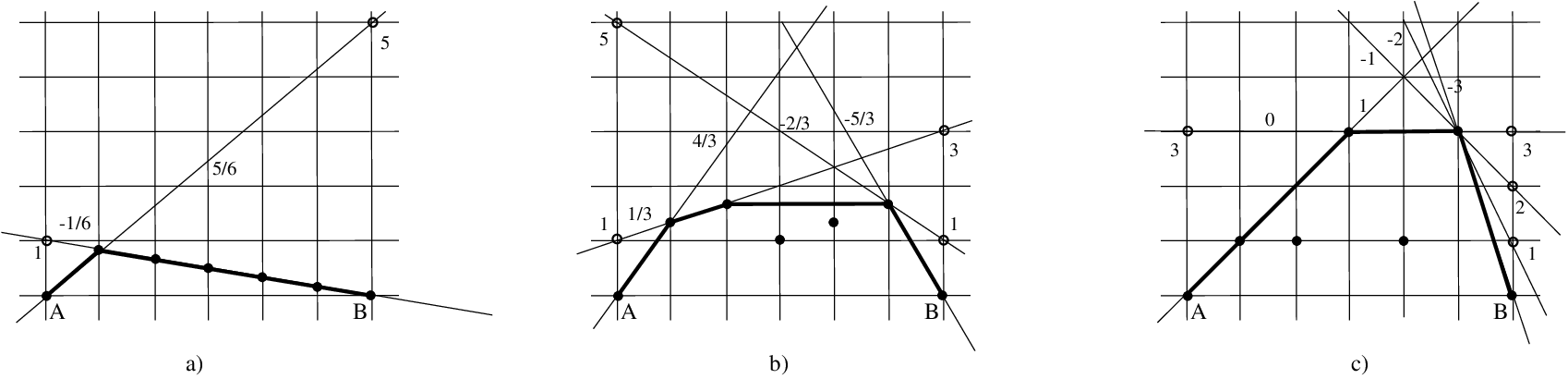}
\caption{Upper convex hulls $C_\tth$ and support lines}
\label{convhulls}
\end{center}
\end{figure}

(ii) Assume now that $\lambda_\tth(f_3)=-2$, and all the other values are the same as in the previous example. Then 
$\chi_\tth(r)=1+(2r \mod d)/d$ for $r\in[1,d-1]$, $\chi_\tth(0)=\chi_\tth(6)=0$,  
and $\tau_\tth=-2/d$. The corresponding upper convex hull $C_\tth$ for $d=6$ is shown in 
Fig.~\ref{convhulls}b). There are two support lines with the slope $c+1/3$, $c\in\Z$,  that do not pass through $A$ or $B$: the line for $c=0$ that passes through $(1,4/3)$ and $(2,5/3)$ and the line for $c=-1$ that passes 
through $(5,5/3)$. The first of the two lines passes through $(0,1)$ and $(6,3)$, so there are three arrows from $k$ and one arrow to $k$; the second line passes through $(0,5)$ and $(6,1)$, so there are five arrows to $k$ and one arrow from $k$. For $c>0$ there are $6c+2$ arrows from $k$, for $c<-1$ there are $-6c+2$ arrows to $k$.

(iii) Assume now that $d=6$,  $\lambda_{\tth}(p_r)=1$, $r=1,2,4$, $\lambda_{tth}(p_r)=3$, $r=3,5$, $\lambda_\tth(f_l)=0$, $l\in[1,3]$. Then $\chi_\tth(r)=\lambda_\tth(p_r)$, $r\in[1,5]$, 
$\chi_\tth(0)=\chi_\tth(6)=0$, $\tau_\tth=0$. The corresponding upper convex hull $C_\tth$ is shown in 
Fig.~\ref{convhulls}c). There are three support lines with the slope $c\in\Z$ that do not pass through $A$ or $B$: the line for $c=0$ that passes through $(3,3)$ and $(5,3)$, the lines for $c=-1$ and $c=-2$ that pass through $(5,3)$.
The first line passes through $(0,3)$ and $(6,3)$, so there
are three arrows in both directions, the second passes through $(0,8)$ and $(6,2)$, so there are eight arrows to $k$ and two arrows from $k$, the third line passes through $(0,13)$ and $(6,1)$, so there are thirteen arrows to $k$ and one arrow from $k$. For $c>0$ there are $6c$ arrows from $k$, for $c<-2$ there are $-6c$ arrows to $k$.
\end{example}

\begin{remark}
The results of Sections~\ref{maincond} and~\ref{combin} remain valid in the case when not all $\ttq_k$ are distinguished, that is, when the product in~\eqref{regpull} is taken over a subset $K'\subset K$ and $\hat\ttf_l$ is a regular function on the  quasi-affine variety $U\subset\AA^r$ defined by non-vanishing of $\ttq_k$, $k\in K\setminus K'$.
In this case the coherent pullback of $\CC$ is an almost-cluster structure on $U$. Note that acting as described in Section~\ref{seedpullback} we can guarantee that in the initial seed $\Psi^*\widehat\Sigma$ all $\hat\ttf_l$ are regular on the whole ambient space $\AA^{r}$; however, this is not true for the exchange coefficients, and hence the variables of the almost-cluster structure are not necessary regular on $\AA^{r}$ after generalized mutations. 
\end{remark}

\section{Compatible pullback of a compatible Poisson bracket}\label{poisson}
Recall that a Poisson bracket $\Poi$ is called {\it compatible\/} with a generalized cluster structure $\CC$ if
$\{f_l,f_m\}=\omega_{lm}f_lf_m$ whenever $f_l$ and $f_m$ belong to the same cluster. Assume that $f_l$, $l\in[n+1,s]$ are frozen variables of $\CC$ and denote 
\[
p_{mr}^*=p_{mr}^{d_m}\prod_{n<l\le s}f_l^{-d_m\{rb_{ml}/d_m\}-d_m\{(d_m-r)b_{lm}/d_m\}},
\]
where $p_{mr}$ are exchange coefficients participating in~\eqref{exchange}. The following sufficient condition of compatibility involving $p_{mr}^*$ and $y$-variables $t_m$ is given in~\cite[Proposition~2.5]{GSVdouble}.

\begin{proposition}\label{combracond}
If all $p_{mr}^*$ are Casimirs of $\Poi$, $\{f_m,t_m\}=\omega_m f_mt_m$, $\omega_m\ne0$, for $m\in[1,n]$, 
and $\{f_l,t_m\}=0$ for $l\ne m$, $l\in[1,s]$, $m\in[1,n]$, then $\Poi$ is compatible with $\CC$.
\end{proposition}

Assume that $\AA^s$ is equipped with a Poisson bracket $\Poi$ that satisfies conditions of 
Proposition~\ref{combracond}, and that there exists a coherent pullback $\widehat\CC$ of $\CC$. Let $\Poip$ be a Poisson bracket on $\AA^r$ that makes the map $\Psi\: (\AA^r,\Poip)\to (\AA^s,\Poi)$ Poisson. Further, let $\tilde\ttt_m$ denote the $y$-variable at vertex $m$ of the quiver $\Psi^*Q$.

\begin{theorem}\label{compbra}
If all $\ttp_{mr}^*$ are Casimirs of $\Poip$, $\{\tth_j,\tilde\ttt_m\}_{\rm p}=0$ for all $j\in J$ and $m\in[1,n]$, and the set
$\ttF\cup\ttH$ forms a log-canonical basis for $\Poip$ then the
bracket $\Poip$ is compatible with $\widehat\CC$.
\end{theorem}

\begin{proof}
We claim that the conditions of the Theorem imply conditions of Proposition~\ref{combracond}. Indeed, 
\begin{multline*}
\{\ttf_l,\tilde\ttt_m\}_{\rm p}=\{\Psi^*f_l\prod_{j\in J}\tth_j^{\lambda_j(f_l)},\tilde\ttt_m\}_{\rm p}\\
=\left\{\prod_{j\in J}\tth_j^{\lambda_j(f_l)},\tilde\ttt_m\right\}_{\rm p}\Psi^*f_l+
\{\Psi^*f_l,\tilde\ttt_m\}_{\rm p}\prod_{j\in J}\tth_j^{\lambda_j(f_l)}.
\end{multline*}
By the conditions of the Theorem, the first term above vanishes. Further, 
by Proposition~\ref{ypullback}, $\tilde\ttt_m=\Psi^*t_m$, so that $\{\Psi^*f_l,\tilde\ttt_m\}_{\rm p}=\Psi^*\{f_l,t_m\}$,
therefore the second term above vanishes for $l\ne m$ and equals  $\omega_m \ttf_m\tilde\ttt_m$ for $l=m$, as required
by  Proposition~\ref{combracond}.
\end{proof}

\section{Regular pullback of the upper cluster algebra}\label{ucla}
Assume that there exists a coherent pullback of the (generalized) cluster structure $\CC$ on $V$ to $\AA^s$. One can define the regular pullback $\widehat{\A}$ of the
upper cluster algebra $\UU$ as the intersection of the rings of Laurent polynomials in almost-cluster variables 
over the ground ring of polynomials in frozen variables  
taken over all almost-clusters. The relation between $\widehat{\A}$ and $\UU$ under very mild conditions is given by the following statement.

\begin{theorem}\label{pullbackucla}
Assume that any seed of $\CC$ is coprime, then $\UU$ is the localization of $\widehat{\A}$ by the multiplicative set generated by $G=\{g_i\}_{i\in I}$.
\end{theorem}

\begin{proof}
Denote $R=\C[\{g_i^{-1}\}_{i\in I}]$. 
To prove the inclusion of the localization of $\widehat{\A}$ in $\UU$ is enough to check that $zh\in\UU$ for any 
$z\in R$ and any $h\in \widehat{\A}$. By definition, in any seed of $\widehat\CC$, $h$ can be written as a Laurent polynomial in almost-cluster variables, which, in turn, are the cluster variables of the corresponding seed of 
$\CC$ multiplied by a Laurent monomial in $G$. It follows that $h\in\UU$. The claim concerning $zh$ follows from the fact that $R$ is a subring of the ground ring of $\UU$.

To prove the opposite inclusion, consider an arbitrary $p\in\UU$. It can be written as a Laurent polynomial in the initial seed $\Sigma$ of $\C$. Consequently, the corresponding Laurent polynomial in the initial seed $\widehat\Sigma$ of 
$\widehat\CC$ differs from $p$ by a Laurent monomial in $G$. It follows that there exists a monomial
$M$ in $G$ such that $Mp$ is a Laurent polynomial in  $\widehat\Sigma$. In a similar way, for any adjacent
seed $\Sigma_m$ there exists a monomial $M_m$ such that $M_mp$ is a Laurent polynomial in  $\widehat\Sigma_m$. Thus, for
$\bar M=M\prod_mM_m$, $\bar Mp$ is simultaneously a Laurent polynomial in $\widehat\Sigma$ and all $\widehat\Sigma_m$. It remains to note that if $P,Q\in\C[\{f_l\}_{l=1}^s, \{g_i^{-1}\}_{i\in I}]$ are coprime then the corresponding 
$\tilde P,\tilde Q$ defined in~\eqref{liftedexch} are coprime as well, and hence any seed of $\widehat\CC$ is coprime, so that the upper bound defined by $\widehat\Sigma$ coincides with the whole $\widehat\A$, see~\cite[Theorem~1.5]{BFZ}.
\end{proof}

The following more restrictive conditions guarantee that the regular pullback to $\AA^s$ of a complete (generalized) cluster structure on $V$ is complete.

\begin{theorem}\label{pullbackcomp}
Assume that there exists a seed $\Sigma=(Q,F\cup G)$ of a regular complete generalized cluster structure $\CC$ on $V$ such that for every $i\in I$ each $y_m$, $m\in[1,s]$, can be written as a rational expression in $\widehat F\cup G\setminus\{g_i\}$ at a generic point of $D_i=\{g_i=0\}$. Then

{\rm (i)} there exists a coherent pullback $\widehat\CC$ of $\CC$ to $\AA^s$;

{\rm (ii)} rational expressions for $y_m$ are Laurent polynomials that do not contain frozen variables in $\widehat F$ in the denominator;

{\rm (iii)} if additionally these rational expressions do not contain variables in $G$ in the denominator then
$\widehat\CC$ is complete.
\end{theorem}

\begin{proof}
(i) It is sufficient to check that the conditions of Theorem~\ref{maincondtheo} are satisfied. Indeed, since the size of the family $\widehat F\cup G\setminus\{g_i\}$ equals the dimension of $D_i$, any algebraic dependence would make impossible to restore $y_m$ at a generic point of $D_i$. 

(ii) Since $\CC$ is complete, every $y_m$ can be written as a Laurent polynomial in any seed $\Sigma$ of $\CC$ with coefficients that are polynomials in frozen variables in $F$ and Laurent polynomials in $G$. This expression can be rewritten in the corresponding seed $\widehat\Sigma$ as a Laurent polynomial with coefficients that are polynomials in frozen variables in $\widehat F$ and $G$ divided my a monomial $L$ in $G$.

(iii) Assume that $L=1$ at some seed $\widehat\Sigma$. We claim that the same holds for any seed 
$\widehat\Sigma'$. Clearly, it is enough to prove that the Laurent representation in the adjacent seed 
$\widehat\Sigma$ obtained  by replacing   an almost-cluster variable $\hat f=\hat f_l$ by $\hat f'$ has this property. 
Let $R$ and $R'$ be the rings of polynomials in the variables of $\widehat\Sigma$ and $\widehat\Sigma'$, respectively. Assume that $y=QM/\hat f^k$, where $k\in\Z$, $Q\in R$ is  not divisible by a monomial and $M$ is a Laurent monomial in almost-cluster variables of $\widehat\Sigma$ excluding $\hat f$ and not containing frozen variables in the denominator. Further, let $y=Q'M'$, where $Q'\in R'$ is not divisible by a monomial and $M'$ is a Laurent monomial in variables of 
$\widehat\Sigma'$. Recall that $\hat f$ and $\hat f'$ satisfy the exchange relation $\hat f\hat f'=P$, where 
$P\in R\cap R'$.
Let $m$ be the minimal non-negative integer such that $\tilde Q=(\hat f')^m \left.Q\right|_{\hat f=P/\hat f'}\in R'$, then
$\tilde Q (\hat f')^{k-m}M=Q'M' P^k$ is an identity in $R'$. Since variables in $\widehat\Sigma'$ are algebraically independent, this identity is only possible if $\tilde Q=Q' P^k$ and $M'=(\hat f')^{k-m}M$, therefore, $M'$ retains the property of not containing frozen variables in the denominator.

Consequently, every $y_m$ can be written as a Laurent polynomial over the ring of polynomials in frozen variables in $\widehat F$ and $G$ at any almost-cluster $\widehat\Sigma'$, so that $\widehat\A$ is complete.
\end{proof}

\begin{example}
Consider the generalized almost-cluster structure $\widehat\GCC_3$ described in Example~\ref{cgacg3mat}. 
A direct check with Maple shows that conditions of Theorem~\ref{pullbackcomp} are satisfied, hence $\widehat\GCC_3$ is complete.
\end{example}

\end{document}